\documentclass[12pt,a4paper]{amsart}
\usepackage{amsmath,amsfonts,amsthm,amsopn,color,amssymb,enumitem}
\usepackage[a4paper]{geometry}
\usepackage{palatino}
\usepackage{graphicx}
\usepackage[colorlinks=true]{hyperref}
\hypersetup{urlcolor=blue, citecolor=red, linkcolor=blue}
\allowdisplaybreaks

\usepackage{relsize}
\usepackage{esint}
\usepackage{verbatim}
\usepackage{mathrsfs}
\usepackage{xcolor}

\newcommand{\e}{\varepsilon}

\newcommand{\R}{\mathbb{R}}
\newcommand{\N}{\mathbb{N}}

\newcommand{\de}{\partial}

\newcommand{\D}{{\mathfrak{D}}}

\newcommand{\hsp}{\hspace{0.2cm}}

\newcommand{\ff}{\mathbb{I}}
\newcommand{\ffz}{\mathbb{I}^\circ}

\newcommand{\sff}{\mathbb{I}}

\newcommand*{\abs}[1]{\left\vert #1\right\vert}
\newcommand*{\norm}[1]{\left\Vert #1\right\Vert}

\newtheorem{theorem}{Theorem}[section]
\newtheorem*{theorem*}{Theorem}
\newtheorem{lemma}[theorem]{Lemma}

\newtheorem{proposition}[theorem]{Proposition}

\theoremstyle{definition}

\title[Non-degeneracy of critical points of the second fundamental form]{Non-degeneracy of critical points of the squared norm of the second fundamental form on manifolds with minimal boundary}

\author{Sergio Cruz-Blázquez}
\address{Sergio Cruz Blázquez, Dipartimento di Matematica, Università degli Studi di Bari, Via E. Orabona, 4, 70125 Bari (Italy)}
\email{sergio.cruz@uniba.it}

\author{Angela Pistoia}
\address{Angela Pistoia, Dipartimento SBAI,  Sapienza Università di Roma, Via Antonio Scarpa 16
	00161 Roma (Italy)}
\email{angela.pistoia@uniroma1.it}
\date\today
 \thanks{ 
 S. Cruz has been partially supported by INdAM – GNAMPA Project 2022 “Fenomeni di blow-up per equazioni non lineari”, E55F22000270001S and by PRIN 2017JPCAPN. A. Pistoia has been partially supported by  GNAMPA, Italy as part of INdAM. }
\begin{document}
\maketitle
\begin{abstract}
Let $(M,\bar g)$ be a compact Riemannian manifold with minimal boundary such that the second fundamental form is nowhere vanishing on $\de M$. We show that for a generic Riemannian metric $\bar g$, the squared norm of the second fundamental form is a Morse function, i.e. all its critical points   are non-degenerate. We show that the generality of this property holds when we restrict ourselves to the conformal class of the initial metric on $M$.
\end{abstract}
\section{Introduction}
Let $(M^n,\bar g)$ be a compact $C^m$ manifold with $m\geq 3$ of dimension $n\geq 3$ with minimal boundary $\de M$, i.e. the boundary mean curvature is identically zero on $\de M$. Assume that the second fundamental form $\sff_{\bar g}$ of $\de M$  is nowhere vanishing.  Our goal is to show that for a generic metric $\bar g$ with said properties, all the critical points of the squared norm of $\sff_{\bar g}$ are non-degenerate.

\medskip 

In a local coordinate system, if we let $h_{ij}$ denote the components of $\sff_{\bar g},$ this function is defined as
\begin{equation}\label{def}
\norm{\sff_{\bar g}}^2=\bar g^{ik}\bar g^{j\ell}h_{ij}h_{k\ell}.
\end{equation}
Consider $\mathscr M^m$ the set of all $C^m$ Riemannian metrics on $M$, and define 
$$\mathscr E^m = \{\bar g\in \mathscr{M}^m:\:\sff_{\bar g}(p)\neq 0\hsp \forall p\in \de M\}.$$
We further assume that $M$ is such that $\mathscr{E}^m$ is a non-empty open subset of $\mathscr{M}^m$. Our first result reads as
\begin{theorem}\label{theorem1} The set 
	\begin{equation*}\begin{split}
			\mathcal{A}=\big\lbrace &g\in \mathscr{E}^m:\:(M,g)\hsp\text{has minimal boundary}\hsp\text{and}\hsp\norm{\sff_g}^2\\&\text{has only non-degenerate critical points on}\hsp\de M\big\rbrace
	\end{split}\end{equation*} 
	is an open dense subset of $\:\left\lbrace g\in \mathscr{E}^m:\:(M,g)\hsp\text{has minimal boundary}\right\rbrace$.
\end{theorem}
A well-known result by Escobar \cite[Lemma 1.1]{Escobar92-2} guarantees that every compact Riemannian manifold $(M,\bar g)$ of dimension $n\geq 3$ with boundary admits a conformal metric with zero boundary mean curvature and whose scalar curvature does not change sign. Moreover, if $\bar g\in \mathscr{E}^m,$ every conformal metric to $\bar g$ is also in $\mathscr{E}^m$. In the context of Conformal Geometry, a natural assumption to work with is that the initial metric $\bar g$ is the one given by Escobar's Lemma, so in order to preserve this geometrical meaning, we also establish the following result:

\begin{theorem}\label{theorem2} The set
\begin{equation*} 
\begin{split}
\mathcal{B}=\big\lbrace &g\in[\bar g]:\:(M,g)\hsp\text{has minimal boundary}\hsp\text{and}\hsp\norm{\sff_g}^2\\ &\text{has only non-degenerate critical points on}\hsp\de M\big\rbrace
\end{split}
\end{equation*}
is an open dense subset of $\:\left\lbrace g\in[\bar g]:\:(M,g)\hsp\text{has minimal boundary}\right\rbrace$.
\end{theorem}
\bigskip

Our results are motivated by the study of compactness of solutions for linear perturbations of the following geometric problem: {\em given smooth functions $K$ and $H$, is there any conformal metric $g$ such that the scalar curvature of $M$ and the boundary mean curvature of $\de M$ with respect to $g$ are $K$ and $H$?}

Considering a conformal metric $g = u^\frac{4}{n-2}\bar g$, with $u$ a smooth and positive function, the problem turns out to be equivalent to finding a positive solution to the doubly critical boundary value problem:
\begin{equation}\label{prob-geo}
	\left\{\begin{array}{ll}-\frac{4(n-1)}{n-2}\Delta_{\bar g} u + S_{\bar g} u =K u^{\frac{n+2}{n-2}}&\mbox{in}\hsp M,\\[0.075in]
	\frac{2}{n-2}\frac{\partial u}{\partial \nu}+H_{\bar g} u = H u^{\frac{n}{n-2}}&\mbox{on}\hsp \partial M,\end{array}\right.
\end{equation}

where $\Delta_{\bar g}$ is the Laplace-Beltrami operator, $S_{\bar g}$ and $H_{\bar g}$ denote the scalar and boundary mean curvatures of $(M,\bar g)$ and $\nu$ is the outward unit normal vector to $\de M.$ By Escobar's result mentioned earlier, via a conformal change of the initial metric $\bar g$ we can start with $H_{\bar g}=0$ and $S_{\bar g}$ of constant sign. 

More precisely, we are concerned with the case of $K<0$ and $H$ of arbitrary sign, which to our knowledge has only been considered by Cruz-Blázquez,  Malchiodi and Ruiz in \cite{CMR}. In this work, the authors introduce the scaling-invariant function $\D_n:\de M\to \R$ given by $$\D_n(x)=\sqrt{n(n-1)}\frac{H(x)}{\sqrt{\abs{K(x)}}},$$and show via a trace inequality that the nature of the associated energy functional depends on whether $\max_{\de M}\D_n<1$ or not. When this condition holds the functional becomes coercive and a solution to \eqref{prob-geo} can be found in the form of a global minimizer on manifolds of non-positive \textit{Yamabe invariant} ($S_{\tilde g}\leq 0$). On the other hand, the functional is no longer bounded from below when $\D_n\geq 1$ at some point on $\de M$. Still, the existence of a positive solution on three dimensional manifolds is recovered using a min-max argument. 

The proof relies on a precise blow-up analysis. First, it is proved that \textit{bubbling of solutions} occurs around boundary points where $\D_n\geq 1$. At blow-up points with $\D_n(p)=1$ the tangential gradient of $\D_n$ must vanish, contradicting a regularity assumption on the set $\{\D_n=1\}$. To deal with the loss of compactness around singular points $p$ with $\D_n(p)>1$, where the limiting profiles are \textit{bubbles}, it is shown that in dimension three they are all \textit{isolated and simple} (see also \cite{DMA, Li-dom} in this regard), allowing for some integral estimates that prevent this behaviour when $S_{\bar g}\leq 0$.

This restriction on the dimension of $M$ is not technical. In fact, three is the maximal dimension for which one can prove that this type of blow-ups are isolated and simple for generic choices of $K$ and $H$. Some partial results are available in higher dimensions (see \cite{AA,BAGOB}), but a general description of the issue is still missing.

\medskip

Motivated by the previous observations, we study the following linear perturbation of \eqref{prob-geo}: 

\begin{equation}\label{prob-per}
	\left\{\begin{array}{ll}-\frac{4(n-1)}{n-2}\Delta_{\bar g} u + S_{\bar g} u =K u^{\frac{n+2}{n-2}}&\mbox{in}\hsp M,\\[0.075in]
		\frac{2}{n-2}\frac{\partial u}{\partial \nu}+\e u = H u^{\frac{n}{n-2}}&\mbox{on}\hsp \partial M,\end{array}\right.
\end{equation}
where $\e>0$ is a small parameter. We address the following questions:
\begin{enumerate}
	\item Do there exist solutions to \eqref{prob-per} that blow-up as $\e\to0$?
	\item If $n\geq 4$, do there exist solutions to \eqref{prob-per} with \textit{clustering} (i.e. non-isolated) and/or \textit{towering} (i.e. non-simple) blow-up points?
\end{enumerate}
In the forthcoming work \cite{CPV} Cruz-Blázquez, Pistoia and Vaira give a positive answer to both questions.

\begin{theorem*}{\cite[Theorem 1.1]{CPV}} Assume
	\begin{enumerate}
		\item $(M^n,\bar g)$ is a manifold of positive Yamabe invariant equipped with the metric given by \cite[Lemma 1.1]{Escobar92-2}, and $4\leq n\leq 7$.
		\item $K<0$ and $H>0$ are constant functions such that $\D_n>1$.
		\item $p\in \de M$ is a non-umbilical point (i.e. $\sff_{\bar g}(p)\neq 0$) and a non-degenerate minimum of $\norm{\sff_{\bar g}}^2$.
	\end{enumerate}
Then $p$ is a clustering blow-up point, that is, for every $k\geq 1$ there exists $\e_k>0$ such that for all $\e\in(0,\e_k)$ the problem \eqref{prob-per} has a solution $u_\e$ with $k$ positive peaks at $p_\e^j$ and $p_\e^j\to p$ as $\e\to 0$.
\end{theorem*}
Notice that Theorem \ref{theorem1} implies that condition (3) holds for a generic choice of a Riemannian metric $\bar g$ on $M$ with $H_{\bar g}=0$. Moreover, this metric can be chosen in the conformal class of that given by Escobar's result by Theorem \ref{theorem2}, and if it is taken close enough to it, the condition $S_{\bar g}>0$ is preserved by continuity.

\medskip 

These results can also be applied to prove the existence of positive solutions to \eqref{prob-per} that blow-up at a single boundary point $p$ which is non-umbilic and a non-degenerate minimum of $\norm{\sff_{\bar g}}^2$ for every dimension $n\geq 4.$ This has been studied by Ghimenti, Micheletti and Pistoia in \cite{GMP1, GMP2} for the scalar-flat case $(K=0)$ in presence of a linear non-autonomus perturbation $\e\gamma u$, for some $\gamma \in C^2(M)$.

\medskip 

The rest of the paper is organized as follows: in section \ref{setting} we describe the setting of the problem. As we will see, the proof of these results relies on a  classical transversality theorem  (see Theorem \ref{TT}) and   is carried out in Section \ref{section-proof}. The proof of the \textit{transversality condition} (see  {\em (T-2)}) is technical and is postponed to Section \ref{section-tech}.
 
\section{Notation and Preliminaries}\label{setting}
%Let $(M,g)$ be a compact Riemannian manifold with boundary. The second fundamental form of the submanifold $\de M$ is defined through the identity
%\begin{equation*}
%	\ff(X,Y)=\nabla_XY-\nabla^{\de M}_XY \hsp \forall\: X,Y\in \mathfrak{X}(\de M),
%\end{equation*}
%where $\nabla^{\de M}$ denotes the intrinsic connection on $\de M$. In components, if $\eta$ is a unit normal vector field along $\de M$, we have
%\begin{equation*}
%\ff(X,Y)=h_{ij}X^iY^j\eta, \hsp\text{for}\hsp i,j=1,\ldots,n-1.
%\end{equation*}
%We call the \textit{mean curvature} of $\de M$ the scalar function resulting from taking the trace of $\ff$, and denote it by $H_g:$ 
%\begin{equation*}
%	H_g = g^{ij}h_{ij}.
%\end{equation*}
%Now, the trace-free part of the second fundamental form is 
%\begin{equation*}
%	\ffz = \ff - \frac{H_g}{n-1}g,
%\end{equation*}
%or in components,
%\begin{equation*}
%	h^\circ_{ij}=h_{ij}-\frac{g^{ab}h_{ab}}{n-1}g_{ij}.
%\end{equation*}
%In this paper we are concerned with the study of the properties of the squared norm of $\ffz$. This quantity is defined as 
%\begin{equation*}
%\norm{\ffz}^2 = g^{ij}g^{k\ell}h^\circ_{ik}h^\circ_{j\ell}
%\end{equation*}
%Straightforward computations show that
%\begin{equation}\label{expan}
%\norm{\ffz}^2=\norm{\ff}^2-\frac{H_g^2}{n-1}=g^{ij}g^{k\ell}h_{ik}h{j\ell}-\frac{\left(g^{ab}h_{ab}\right)^2}{n-1}
%\end{equation}
In what follows we often identify a point in the manifold with its Fermi coordinates $(x_1,\ldots,x_{n-1},t)$, which we now recall.

Take $\bar p\in \de M$, and let $(x_1,\ldots,x_{n-1})$ be normal coordinates on the $(n-1)-$dimensional manifold $\de M$, valid in a geodesic ball $B^{\bar g}_R(\bar p)\subset \de M$, with $R>0$ small enough. Let $\gamma(t)$ be the geodesic leaving from $\bar p$ in the orthogonal direction to $\de M$ and parametrized by arc length. Then $(x_1,\ldots,x_{n-1},t)$ are the so-called Fermi coordinates at $\bar p$, where $(x_1,\ldots,x_{n-1})\in B_R(0)$ and $t\in [0,T]$, for some $T>0$ small. 

We adopt the following notation for derivatives:
\begin{align*}
	\de_s = \frac{\de}{\de x_s} \hsp\text{ for}\hsp s=1,\ldots,n-1,\hsp \text{and}\hsp \de_t = \frac{\de}{\de t}
\end{align*}
For convenience, we list the properties of the Fermi coordinates that we will use in our proofs. These and other properties of the Fermi coordinates can be found in \cite[\textsection 5]{lee-book}.

\begin{proposition}{\bf Properties of the Fermi Coordinates.}\label{fermi} Let $(x_1,\ldots,x_{n-1},t)$ be the Fermi coordinates on $B_R(0)\times[0,T]$ with respect to the metric $\bar g$ centered at $\bar p$. Then, $\bar g(x,0)$ satisfies:
\begin{enumerate}
	\item[(i.)] $\bar g_{ij}(0,0)=\delta_{ij}$ for every $1\leq i,j\leq n-1$.
	\item[(ii.)] $\de_s \bar g_{ij}(0,0)=0$ for every $1\leq i,j\leq n-1$.
	\item[(iii.)] $\bar g_{in}(x,0)=0$ for every $i=1,\ldots,n-1$ and $x\in B_R(0)$.
	\item[(iv.)] $\bar g_{nn}(x,0)=1$ for every $x\in B_R(0)$.
\end{enumerate} 
\end{proposition}
For $t\in[0,T]$, by $h_{ij}(x,t)$ we understand the second fundamental form of the submanifold
$$\de M_t = \{(x,t):\:x\in \de M\}.$$
A result by Escobar from \cite{Escobar92} gives the following relation:
\begin{equation}\label{esc} 
	h_{ij}(x,t)=-\frac12\de_t \bar g_{ij}(x,t).
\end{equation} 
We denote by $\mathscr{S}^m$ the Banach space of $C^m$ symmetric covariant 2-tensors on $M$ equipped with the norm $\norm{\cdot}_m$, which is defined as follows: fix $\{(V_\lambda):\:\lambda\in\Lambda\}$ a finite covering of $M$ such that the closure of every $V_\lambda$ is contained in an open coordinate neighborhood $(U_\lambda,\psi_\lambda)$, and such that if $V_\lambda\cap\de M\neq \emptyset$, then $\psi_\lambda = (x_1,\ldots,x_{n-1},x_n)$ are Fermi Coordinates. Then, if $k\in\mathscr{S}^m$ and $k_{ij}$ are its components with respect to the coordinates on $V_\lambda$, we define
\begin{equation*}
\norm{k}_m=\sum_{\lambda\in \Lambda}\sum_{\abs{\alpha}\leq m}\sum_{i,j=1}^n\sup_{\phi_\lambda(V_\lambda)}\frac{\de^\alpha k_{ij}}{\de x_1^{\alpha_1}\cdots\de x_{n}^{\alpha_n}}
\end{equation*}
We will denote by $\mathscr{B}_\rho(k_0)$ the ball of radius $\rho>0$ and center $k_0$ in $\mathscr{S}^m$ with respect to the norm $\norm{\cdot}_m$.

\medskip 

In our case, due to the restriction to metrics with zero boundary mean curvature, we need to apply a transversality theorem for general Banach manifolds. The following result is stated and proved in Henry's book \cite[Theorem 5.4]{Henry-book}.

\begin{theorem}{\bf Transversality Theorem}\label{TT} Let $m$ be a positive integer, $X,Y,Z$ Banach manifolds of class $C^m$, $A\subset X\times Y$ an open set, $F:A\to Z$ a $C^m$ map and $\xi\in Z$. Assume that for each $(x,y)\in F^{-1}(\{0\})$:
	\begin{enumerate}
		\item[(T-1)] $D_xF(x,y):T_xX\to T_\xi Z$ is semi-Fredholm with index $<m$,
		\item[(T-2)] $DF(x,y):T_xX\times T_yY\to T_\xi Z$ is surjective, and
		\item[(T-3)] The map $\pi\circ i:F^{-1}(\{0\})\to Y$ defined by $\pi\circ i(x,y)=y$ is $\sigma-$proper, that is, there exists a countable collection $\{M_j:j\in\N\}$ such that $F^{-1}(\{0\})=\cup_{j=1}^{+\infty}M_j$ and $\pi\circ i\vert_{M_j}:M_j\to Y$ is proper for every $j$. 
	\end{enumerate}
	Then the set
	\begin{equation*}
		\Omega = \left\lbrace y\in Y:\: \xi \hsp\text{is a critical value of } f(\cdot,y)\right\rbrace
	\end{equation*}
	is a meager set in $Y$ and, if $\pi\circ i$ is proper, $\Omega$ is also closed.
\end{theorem}
Now, we introduce the maps to which we will apply the Transversality Theorem. Remember that $(M,\bar g)$ is a compact Riemannian manifold with boundary and $\bar g$ is such that $H_{\bar g}=0$ and $\sff_{\bar g}(p)\neq 0$ for every $p\in \de M$. First, we define
\begin{equation}
N = \left\lbrace k\in\mathscr{E}^m:\:H_{\bar g+k}=0 \right\rbrace
\end{equation}
\begin{lemma} $N$ is a submanifold of $\mathscr{E}^m$.
\end{lemma}
\begin{proof}
Clearly, $N=\tilde H^{-1}(\{0\})$, where $\tilde H:\mathscr{S}^m(M) \to C(\de M, \R)$ is the map given by $\tilde H(k)=H_{\tilde g+k}$. To show that $N$ is a submanifold, it is enough with proving that $0$ is a regular value of $\tilde H$.

\medskip 
Take $k_0\in \mathscr{S}^m$ such that $\tilde H(k_0)=H_{\tilde g+k_0}=0.$ The derivative of $\tilde H$ at the point $k_0$ is the map
\begin{equation*}
	\begin{split}
		D\tilde H(k_0):\mathscr{S}^m(M) &\to C(\de M,\R) \\
		k&\to D\tilde H(k_0)[k]=DH_{\bar g+k_0}[k]
	\end{split}
\end{equation*}
We will show that $D\tilde H(k_0)$ is surjective. For a small enough $\e>0$ and $\tau\in(0,\e)$, 
\begin{align*}
	D\tilde H(k_0)[k]=\left.\frac{d}{d\tau}\right\vert_{\tau=0}\tilde H(k_0+\tau k)=\left.\frac{d}{d\tau}\right\vert_{\tau=0} (\bar g+k_0+\tau k)^{ij}(\ff_{\bar g+k_0+\tau k})_{ij}.
\end{align*}
If $\e$ is small enough, the following expansion holds:
\begin{equation}\label{expansion}
	(\bar g+k_0+\tau k)^{ij}=(g_0+\tau k)^{ij}=g_0^{ij}-\tau(g_0^{-1}kg_0^{-1})_{ij}+\sum_{\lambda=2}^{+\infty}\tau^\lambda(-1)^\lambda\left((g_0^{-1}k)^\lambda g_0^{-1}\right)_{ij},
\end{equation}
where we are calling $g_0=\bar g+k_0$. Denote by $h^0_{ij}$ the coefficients of the second fundamental form of $g_0$. Then, using \eqref{esc},
\begin{equation}\label{cond}
	\begin{split}
		D\tilde H(k_0)[k]&= -(g_0^{-1}kg_0^{-1})_{ij}h^0_{ij}+g_0^{ij}\left.\frac{d(\ff_{g_0+\tau k})_{ij}}{d\tau}\right\vert_{\tau=0}
\\ &=-g_0^{i\ell}g_0^{kj}k_{\ell k}h^0_{ij}-\frac{1}{2}g_0^{ij}\left.\frac{dk_{ij}}{dt}\right\vert_{t=0} \\
	%	&=-\langle k,\ff_{g_0} \rangle_{g_0} -\frac{1}{2}g_0^{ij}\left.\frac{dk_{ij}}{dt}\right\vert_{t=0}
	\end{split}
\end{equation}
Therefore, we are led to show that, for every $\mathfrak h\in C(\de M,\R)$, there exists $k\in\mathscr{S}^m$ such that
\begin{equation}\label{surj-1}
-g_0^{i\ell}g_0^{kj}k_{\ell k}h^0_{ij} -\frac{1}{2}g_0^{ij}\left.\frac{dk_{ij}}{dt}\right\vert_{t=0} = \mathfrak h
\end{equation}
Take $k$ such that $k_{ij}=0$ if $(i,j)\neq (1,1)$, and $k_{11}=\bar k$. Then \eqref{surj-1} reduces to
\begin{equation}\label{surj-2}
\left.\frac{d\bar k(x,t)}{dt}\right\vert_{t=0}+ f(x,t)\bar k(x,t)=\frac{-\mathfrak h(x)}{g_0^{11}(x,t)},
\end{equation}
where 
\begin{equation}\label{f}
f(x,t)=\frac{2g_0^{i1}(x,t)g_0^{j1}(x,t)h^0_{ij}(x,t)}{g_0^{11}(x,t)}.
\end{equation}
Solutions of \eqref{surj-2} are explicit and given by
\begin{equation*}
\bar k(x,t)= \exp\left(-\int_0^tf(x,\tau)d\tau\right)\left(c(x)-\int_0^t\frac{\mathfrak h(x)\exp\left(\int_0^{\tau'}f(x,\tau)d\tau\right)}{g_0^{11}(x,\tau')}d\tau'\right),
\end{equation*}
where $c(x)$ is any continuous function.
\end{proof}
We introduce the $C^1$ map related to the proof of Theorem \ref{theorem1}: let $X=\de M$ and fix $\bar p\in \de M$ such that $\sff(\bar p)\neq 0$,  $Y=N$ and $A=\mathscr{B}_\rho(0)\cap N\times B_R(\bar p)$ for $\rho,R>0$ small enough. Define 
\begin{equation}\label{mappa}
	\begin{split}
		F:A &\to \R^{n-1} \\
		F(p,k)&=\nabla_x\norm{\ff_{\bar g+k}(p)}^2.
	\end{split}
\end{equation}
\begin{lemma}\label{open-A} The set
\begin{equation*}
\begin{split}
\mathcal A = \big\lbrace g\in N:\:\text{all the critical points of}\hsp\norm{\sff_g}^2\hsp{on}\hsp\de M\hsp\text{are non-degenerate}\big\rbrace	
\end{split}
\end{equation*}
is an open set of $N$.
\end{lemma}
\begin{proof}
Fix $\hat g\in \mathcal A$. By the compactness of $\de M$, the critical points of $p\mapsto \norm{\sff_{\hat g}(p)}^2$ form a finite collection $\{\xi_1,\ldots,\xi_m\}$. 

\medskip

Around any $i=1,\ldots,m$ we can take Fermi coordinates, and consider the function $F$ defined in \eqref{mappa}. The Implicit Function Theorem gives us $\rho_i,R_i>0$ small enough and a function $\chi_i$ such that in $B_{R_i}(0)\times\mathscr{B}_{\rho_i}(0)$ the level set $\{F=0\}$ can be seen as the graph of the function $\chi_i(k)$. Define
$$B=\bigcup_{i=1}^m\:\{p\in\de M:\:d(p,\xi_i)<R_i\}$$
We claim that the metric $\hat g+k$ is such that $\norm{\sff_{\hat g+g}}^2$ has no critical points in \linebreak $\de M\setminus B$ for any $g\in\mathscr{B}_\rho(0)$ provided $\rho$ is small enough.

\medskip

Suppose by contradiction that there exists $\rho_k\to 0^+$, $g_k\in \mathscr{B}_{\rho_k}(0)$ and $p_k\in \de M\setminus B$ such that $F(p_k,g_k)=0.$ By compactness, there exists a subsequence of $p_k$ such that $p_k\to p\in \de M\setminus B$, and then by continuity $F(0,p)=0$, which contradicts that the only critical points of $\norm{\sff_{\hat g}}^2$ are $\xi_1,\ldots,\xi_m$, and $\xi_i\in B$ for every $i$. This proves that the set $$\big\lbrace g\in \mathscr{M}^m:\:\text{all the critical points of}\hsp\norm{\sff_g}^2\hsp{on}\hsp\de M\hsp\text{are non-degenerate}\big\rbrace$$ is open in $\mathscr{M}^m$. Therefore, $\mathcal A$ is open in $N$ with respect to the induced topology.
\end{proof}
Consider now a conformal metric $\hat g = e^f\bar g$ for some smooth function $f$. Then, the second fundamental form of $\de M$ with respect to $\hat g$, $\hat h_{ij}$, satisfies the relation 
\begin{equation}\label{sff-conf}
	e^{-f}\hat h_{ij}=h_{ij}+\frac{\de f}{\de \nu}\bar g_{ij}
\end{equation}
Consequently,
\begin{equation}\label{curvmedia}
	\begin{split}
		H_{\hat g} = e^{-f}\left(\frac{\de f}{\de \nu}+H_{\bar g}\right).
	\end{split}
\end{equation}
In view of \eqref{curvmedia}, if $H_{\bar g}=0$, then $H_{\hat g}=0$ if and only if $\frac{\de f}{\de\nu}=0$. This leads us to consider the set
\begin{equation}
	\Theta = \left\lbrace f\in C^m(M):\frac{\de f}{\de \nu}=0\right\rbrace.
\end{equation}
Moreover, taking norms in \eqref{sff-conf}
\begin{equation}\label{norm-sff-conf}
	\norm{\sff_{\hat g}}^2 = e^{-2f}\norm{\sff_{\bar g}}^2
\end{equation} 
for every $f\in\Theta$. From this equation we immediately see that, if $g\in\mathscr{E}^m$ and $f\in\Theta$, then $e^{2f}g\in \mathscr{E}^m$, as we had anticipated. Moreover, notice that $\Theta$ is a Banach space, as it is a closed subspace of the Banach space $C^m(M)$. 

\medskip

Let $X=\de M$ and fix $\bar p\in\de M$ such that $\sff(\bar p)\neq 0$, $Y=\Theta$ and $A=B_R(\bar p)\times B_\rho(0)\subset \de M\times \Theta$ for $\rho,$ $R>0$ small enough. The function that we will apply the transversality theorem for Banach spaces is defined as follows:
\begin{equation*}
	\begin{split}
		G:A&\to\R^{n-1} \\
		G(p,u)&=\nabla_x\norm{\sff_{e^{2u}\bar g}(p)}^2 
	\end{split}
\end{equation*}
An analogous reasoning to that of Lemma \ref{open-A} leads to the following result:
\begin{lemma}\label{open-B} The set
	\begin{equation*}
		\begin{split}
			\mathcal B = \big\lbrace u\in \Theta:\:\text{all the critical points of}\hsp\norm{\sff_{e^{2u}\bar g}}^2\hsp{on}\hsp\de M\hsp\text{are non-degenerate}\big\rbrace	
		\end{split}
	\end{equation*}
	is an open set of $\Theta$.
\end{lemma}

\section{Proof of Theorems \ref{theorem1} and \ref{theorem2}}\label{section-proof}
To prove Theorem \ref{theorem1}, we apply Theorem \ref{TT} to the map $F$ given by \eqref{mappa}. Here we have $X=\de M$, which we identify with $\R^{n-1}$ via the Fermi coordinates, $Y=N$ and $A=B_r(\bar p)\times\mathscr{B}_\rho(0)\cap N$, with $\bar p\in \de M$ being such that $\sff(p)\neq 0.$ 

\medskip 

Since $\text{dim }T_{\bar p}\de M=\text{dim }\R^{n-1}=n-1$, any linear operator $T_{\bar p}\de M\to \R^{n-1}$ is Fredholm of degree $0$ (see \cite[Ex. (1) pp. 74]{Henry-book}), so (T-1) holds. The proof of (T-2) is collected in Proposition \ref{sur-F}. As for (T-3), the proof provided in \cite{Ghimenti13} holds with the addition of a few remarks:

\medskip

We write $F^{-1}(0)=\cup_{j=1}^{k_0} C_j$, with $$C_j = \left(\overline{B_{R-\frac1j}(\bar p)}\times \overline{\mathscr{B}_{\rho-\frac1j}(0)}\cap N\right)\cap F^{-1}(0).$$
Take $(x_i,k_i)$ a sequence in $C_j$ such that $k_i\to k$ in $\mathscr{S}^m.$ By continuity of $H_{g}$ with respect to $g$, we have $H_{\bar g+k}=0$, so $k\in \overline{\mathscr{B}_{\rho-\frac1j}(0)}\cap N$. Moreover, by compactness, there exists a subsequence of $x_i$ converging to some point $x\in\overline{B_{R-\frac1j}(\bar p)}$. Finally, since $F$ is continuous, $F(x,k)=0$ so $(x,k)\in C_j.$

\medskip

Then, by Theorem \ref{TT}, the set 
\begin{align*}
	\Omega(\xi,\rho)&=\big\lbrace k\in \mathscr{B}_\rho(p)\cap N:\:D_kF(x,k):\R^{n-1}\to\R^{n-1}\hsp \text{is invertible} \\&\hspace{1in}\text{at any}\hsp(x,k)\hsp\text{with}\hsp F(x,k)=0 \big\rbrace \\&=\big\lbrace k\in\mathscr{B}_\rho(p)\cap N:\: \text{the critical points of}\hsp \norm{\sff_{\bar g+k}}^2\\&\hspace{1in}\text{in}\hsp B_R(\xi)\subset \de M\hsp\text{are non-degenerate} \big\rbrace
\end{align*}
is a residual subset of $\mathscr{B}_\rho(0)\cap N$. Since $\de M$ is compact, we can take a finite covering $\big\lbrace B_R(p_i):p_i\in\de M\hsp\text{for every}\hsp i=1,\ldots,k_1\big \rbrace$. For every $B_R(p_i)$ there exists an open, dense subset $\Omega_{x_i,\rho}$ of $\mathscr B_\rho(0)\cap N$ such that all the critical points of $\norm{\sff_{g+k}}^2$ in $B_R(p_i)$ are non-degenerate for any $k\in \Omega(p_i,\rho)$. Define 
$$\Omega(\rho)=\cap_{i=1}^{k_1}\Omega(p_i,\rho),$$
then $\Omega(\rho)$ is a residual subset of $\mathscr{B}_\rho(0)\cap N$ such that all the critical points of $\norm{\sff_{g+k}}^2$ in $\de M$ are non-degenerate for any $k\in \Omega(\rho).$ Consequently, if $\rho$ is small enough, there exists $\bar k\in \Omega(\rho)$ such that all the critical points of $\norm{\sff_{\bar g+\bar k}}^2$ in $\de M$ are non-degenerate. Thus, the set $\mathcal{A}$ is dense and is open by Lemma \ref{open-A}.

\medskip

The same reasoning allows us to demonstrate Theorem \ref{theorem2}, with the fundamental differences being in the proof of (T-2) (see Proposition \ref{sur-G}). In this case, we apply Theorem \ref{TT} to the map \ref{mappa2}, with $X=\de M$, $Y=\Theta$ and $A=B_R(\bar p)\times \mathfrak B_\rho(0)\cap \Theta$. (T-1) is clear. Concerning (T-3), it is sufficient to notice that if $u_i\to u$ is a convergent sequence in $C^2(M)$ such that $\nabla u_i \cdot \nu=0$ for every $i$, then $\nabla u\cdot \nu=0$, so the limit stays on $\Theta.$ $\mathcal B$ is open by Lemma \ref{open-B}.
\section{The Transversality Condition}\label{section-tech}
\begin{proposition}\label{sur-F}
The map
\begin{equation*}
	\begin{split}
		D_kF(p_0,k_0):\R^{n-1}\times T_{k_0}N&\to \R^{n-1} \\
		(x,k)&\to D_kF(p_0,k_0)[k]+D_xF(p_0,k_0)[x]
	\end{split}
\end{equation*}
is surjective for every $(p_0,k_0)\in \de M\times N$ such that $F(p_0,k_0)=0$.
\end{proposition}
\begin{proof}
Clearly, it is sufficient to prove that the map $k\mapsto D_kF(p_0,k_0)$ is surjective from $T_{k_0}N$ on $\R^{n-1}$.

First, we differentiate the identity $g_{ik}g^{kj}=\delta_i^j$ to obtain 
\begin{equation}\label{dif}
	\de_s g^{ij}=-g^{ia}g^{jb}\de_sg_{ab}.
\end{equation}
Using \eqref{def}, \eqref{esc} and \eqref{dif}, we see that
\begin{align*}
\de_s\norm{\ffz_{g+k}}^2=&\de_s\bigg[\frac14(g+k)^{ij}(g+k)^{k\ell}\de_t(g+k)_{ik}\de_t(g+k)_{j\ell}\bigg] \\ = &\frac{1}{4}\bigg(-(g+k)^{ia}(g+k)^{jb}\de_s(g+k)_{ab}(g+k)^{k\ell}\de_t(g+k)_{ik}\de_t(g+k)_{j\ell} \\ &- (g+k)^{ij}(g+k)^{kc}(g+k)^{\ell d}\de_s(g+k)_{cd}\de_t(g+k)_{ik}\de_t(g+k)_{j\ell} \\
&+(g+k)^{ij}(g+k)^{k\ell}\de^2_{st}(g+k)_{ik}\de_t(g+k)_{jl} \\ &+(g+k)^{ij}(g+k)^{k\ell}\de_t(g+k)_{ik}\de^2_{st}(g+k)_{jl}\bigg)
\end{align*}
By using \eqref{expansion}, we obtain:
\begin{align*}
&D_k\de_s \norm{\ff_{g+k}}^2\vert_{(p_0,k_0)}[k] = D_\mathfrak g\de_s \norm{\ff_{\mathfrak g}}^2\vert_{(p_0,g_0)}[k]=D_\mathfrak g \de_s \left.\bigg(\frac14\mathfrak g^{ij}\mathfrak g^{k\ell}\de_t\mathfrak g_{ik}\de_t\mathfrak g_{j\ell}\bigg)\right\vert_{p_0,g_0} \\
= &\frac14\bigg((g_0^{-1}kg_0^{-1})_{ia}g_0^{jb}\de_s(g_0)_{ab}g_0^{k\ell}\de_t(g_0)_{ik}\de_t(g_0)_{j\ell} +g_0^{ia}(g_0^{-1}kg_0^{-1})_{jb}\de_s(g_0)_{ab}g_0^{k\ell}\de_t(g_0)_{ik}\de_t(g_0)_{j\ell} \\
&-g_0^{ia}g_0^{jb}\de_sk_{ab}g_0^{k\ell}\de_t(g_0)_{ik}\de_t(g_0)_{j\ell} +g_0^{ia}g_0^{jb}\de_s(g_0)_{ab}(g_0^{-1}kg_0^{-1})_{k\ell}\de_t(g_0)_{ik}\de_t(g_0)_{j\ell} \\
&-g_0^{ia}g_0^{jb}\de_s(g_0)_{ab}g_0^{k\ell}\de_tk_{ik}\de_t(g_0)_{j\ell}-g_0^{ia}g_0^{jb}\de_s(g_0)_{ab}g_0^{k\ell}\de_t(g_0)_{ik}\de_tk_{j\ell} \\
&+(g_0^{-1}kg_0^{-1})_{ij}g_0^{kc}g_0^{\ell d}\de_s(g_0)_{cd}\de_t(g_0)_{ik}\de_t(g_0)_{j\ell} + g_0^{ij}(g_0^{-1}kg_0^{-1})_{kc}g_0^{\ell d}\de_s(g_0)_{cd}\de_t(g_0)_{ik}\de_t(g_0)_{j\ell} \\
&+g_0^{ij}g_0^{kc}(g_0^{-1}kg_0^{-1})_{\ell d}\de_s(g_0)_{cd}\de_t(g_0)_{ik}\de_t(g_0)_{j\ell} - g_0^{ij}g_0^{kc}g_0^{ld}\de_sk_{cd}\de_t(g_0)_{ik}\de_t(g_0)_{j\ell} \\ &-g_0^{ij}g_0^{kc}g_0^{ld}\de_s(g_0)_{cd}\de_tk_{ik}\de_t(g_0)_{j\ell} - g_0^{ij}g_0^{kc}g_0^{ld}\de_s(g_0)_{cd}\de_t(g_0)_{ik}\de_tk_{j\ell} \\ &- (g_0^{-1}kg_0^{-1})_{ij}g_0^{k\ell}\de^2_{st}(g_0)_{ik}\de_t(g_0)_{jl}-g_0^{ij}(g_0^{-1}kg_0^{-1})_{k\ell}\de^2_{st}(g_0)_{ik}\de_t(g_0)_{j\ell} \\ &+g_0^{ij}g_0^{k\ell}\de^2_{st}k_{ik}\de_t(g_0)_{jl} +g_0^{ij}g_0^{k\ell}\de^2_{st}(g_0)_{ik}\de_tk_{jl} - (g_0^{-1}kg_0^{-1})_{ij}g_0^{k\ell}\de_{t}(g_0)_{ik}\de^2_{st}(g_0)_{jl}\\ &-g_0^{ij}(g_0^{-1}kg_0^{-1})_{k\ell}\de_{t}(g_0)_{ik}\de^2_{st}(g_0)_{j\ell} +g_0^{ij}g_0^{k\ell}\de_{t}k_{ik}\de^2_{st}(g_0)_{jl} +g_0^{ij}g_0^{k\ell}\de_{t}(g_0)_{ik}\de^2_{st}k_{jl}\bigg)\bigg\vert_{p_0,g_0}
%\\
%&+\frac{1}{n-1}\bigg(-(g_0^{-1}kg_0^{-1})_{\alpha\beta}\de_t(g_0)_{\alpha\beta}\times\left[-(g_0)^{\alpha p}(g_0)^{\beta q}\de_s(g_0)_{pq}\de_t (g_0)_{\alpha\beta}+(g_0)^{\alpha\beta}\de^2_{st}(g_0)_{\alpha\beta}\right] \\
%&+g_0^{\alpha\beta}\de_tk_{\alpha\beta}\times\left[-(g_0)^{\alpha p}(g_0)^{\beta q}\de_s(g_0)_{pq}\de_t (g_0)_{\alpha\beta}+(g_0)^{\alpha\beta}\de^2_{st}(g_0)_{\alpha\beta}\right] \\ 
%&+g_0^{\alpha\beta}\de_t(g_0)_{\alpha\beta} \times \left[\cdots\right]\bigg)
\end{align*}

%\smallskip 

%We know

%\begin{align*}
%g_0^{ij}\de_t(g_0)_{ij}&=0,\\
%(g_0^{-1}kg_0^{-1})_{ij}&=g_0^{ia}g_0^{bj}k_{ab} \\
%-g_0^{i\ell}g_0^{kj}k_{\ell k}\de_t(g_0)_{ij}&=\frac{1}{2}g_0^{ij}\de_tk_{ij}
%\end{align*}

%\medskip

The ontoness of the map $k\to D_kF(p_0,k_0)[k]$ does not depend on the choice of the coordinates, so we choose Fermi coordinates $(x_1,\ldots,x_{n-1},t)$ in $\de M$ with the metric $g_0$ centered at $p_0$. Then, applying Proposition \ref{fermi},
\begin{equation*}
	\begin{split}
D_k\de_s \norm{\ff_{g+k}}^2\vert_{(p_0,k_0)}[k] &= \frac14\bigg(4k_{ij}\:\de_s(g_0)_{ik}\de_t(g_0)_{j\ell}\de_t(g_0)_{k\ell}-2\de_sk_{ij}\:\de_t(g_0)_{ik}\de_t(g_0)_{jk} \\ &+2k_{ij}\:\de_t(g_0)_{ik}\de_t(g_0)_{j\ell}\de_s(g_0)_{k\ell}-4\de_tk_{ij}\:\de_t(g_0)_{ik}\de_s(g_0)_{jk}\\&-4k_{ij}\:\de_t(g_0)_{ik}\de^2_{st}(g_0)_{jk} +2\de^2_{st}k_{ij}\:\de_t(g_0)_{ij}+2\de_tk_{ij}\:\de^2_{st}(g_0)_{ij}\bigg) 
%\\&+\frac{1}{n-1}\bigg[-k_{ij}\de_t(g_0)_{ij}\left(-\de_s(g_0)_{ij}\de_t(g_0)_{ij}+\de^2_{st}(g_0)_{ii}\right) \\ &+\de_tk_{ii}\left(-\de_s(g_0)_{ij}\de_t(g_0)_{ij}+\de^2_{st}(g_0)_{ii}\right)\bigg]
\bigg\vert_{p_0,g_0}
	\end{split}
\end{equation*}

%Moreover, assume that $k_{ij}=0$ if $(i,j)\neq (1,1)$, and call $k_{11}=\bar k.$ We have:
%\begin{align*}
%D_k\de_s \norm{\ff_{g+k}}^2\vert_{(p_0,k_0)}[k] = \frac14\bigg( 6\bar k\de_s(g_0)_{1j}\de_t(g_0)_{1k}\de_t(g_0)_{jk} - 2\de_s \bar k\de_t(g_0)_{1k}\de_t(g_0)_{ik} \\
%-4\de_t\bar k\de_s(g_0)_{1j}\de_t(g_0)_{1j}-4\bar k \de^2_{st}(g_0)_{i1}\de_t(g_0)_{i1}+2\de^2_{st}\bar k\de_t(g_0)_{11}+\de_t\bar k \de^2{st}(g_0)_{11}\bigg) \\ +\frac{1}{n-1}\bigg[-\bar k \de_t(g_0)_{11}\left(-\de_s(g_0)_{11}\de_t(g_0)_{11}+\de^2_{st}(g_0)_{11}\right)\\+\de_t\bar k\left(-\de_s(g_0)_{11}\de_t(g_0)_{11}+\de^2_{st}(g_0)_{11}\right)\bigg]
%\end{align*}

%\begin{equation}\label{step2}
%\begin{split}
%	D_k\de_s \norm{\ff_{g+k}}^2\vert_{(p_0,k_0)}[k] = \frac14\bigg( - 2\de_s \bar k\de_t(g_0)_{1k}\de_t(g_0)_{1k} \\
%-4\bar k \de^2_{st}(g_0)_{i1}\de_t(g_0)_{i1}+2\de^2_{st}\bar k\de_t(g_0)_{11}+\de_t\bar k \de^2_{st}(g_0)_{11}\bigg) \\ +\frac{1}{n-1}\bigg[-\bar k \de_t(g_0)_{11}\de^2_{st}(g_0)_{11}+\de_t\bar k\de^2_{st}(g_0)_{11}\bigg]
%\end{split}
%\end{equation}
Now we make a choice for $k$. To guarantee that $k$ lies in $T_{k_0}N$, it must satisfy the following differential equation:
\begin{equation}\label{eq-k}
\sum_{k\ell}\left(\sum_{ij}g_0^{i\ell}g_0^{kj}h^0_{ij}k_{\ell k} +\frac 12 g_0^{\ell k}\de_t k_{\ell k}\right) = 0
\end{equation}
In view of \eqref{eq-k}, we can choose its components $k_{\ell k}$ to satisfy
\begin{equation}\label{eq-kij}
\sum_{ij}g_0^{i\ell}g_0^{kj}h^0_{ij}k_{\ell k} +\frac 12 g_0^{\ell k}\de_t k_{\ell k} = 0
\end{equation}
Similarly to \eqref{surj-2}, equation \eqref{eq-kij} can be integrated and its solutions explicitly calculated: define 
\begin{equation*}
	f_{k\ell}(x,t)=\frac{2g_0^{ik}(x,t)g_0^{j\ell}(x,t)h^0_{ij}(x,t)}{g_0^{k \ell}(x,t)}
\end{equation*}
Then 
\begin{equation}\label{k}
k_{k \ell}(x,t)= c_{k\ell}(x)\exp\left(-\int_0^tf_{k\ell}(x,\tau)d\tau\right)
\end{equation}
solves \eqref{eq-kij}. By choosing $c_{k\ell}(0)=0$, it is easy to see that
\begin{equation}\label{prop-k1}
	\begin{split}
		k_{k\ell}(0,0)&=0, \\
		\de_t k_{k\ell}(0,0) &= 0, \\
		\de_s k_{k\ell}(0,0)&=\de_s c_{k\ell}(0), \\
		\de^2_{st}k_{k\ell}(0,0) &= -f_{k\ell}(0,0)\de_s c_{k\ell}(0),
	\end{split}
\end{equation}
with $\de_s c_{k\ell}(0)$ to be determined later. With this choice of $k$, we have
\begin{equation}\label{simpl}
\begin{split}
D_k\de_s \norm{\ff_{g+k}}^2\vert_{(p_0,k_0)}[k] &= -2\sum_{ij}\de_s c_{ij}(0)\bigg(\sum_{k}h^0_{ik}h^0_{jk} -2(h^0_{ij})^2\bigg) 
%\\ &=-2A_s\bigg(\sum_{ijk}h^0_{ik}h^0_{jk} -\sum_{ij}h^0_{ij}^2\bigg)
\end{split}
\end{equation}
We claim that:
	\begin{equation}\label{c1}
		\hbox{If $\sff_{g_0}\not=0$ there exist indices $i^*,j^*$ such that}\quad  \sum_{k}h^0_{i^*k}h^0_{j^*k} -2(h^0_{i^*j^*})^2\not=0.
	\end{equation}
	By contradiction, suppose
	$$2(h^0_{ij})^2=\sum_\kappa (h^0_{ik})^2$$
	for any $i,j$. Then, adding over $j$ and $i$
	%$$2\sum_{j}h^0_{ij}^2=n\sum_\kappa h^0^2_{ik},$$
	%and adding also over $i$
	$$2\sum_{ij}(h^0_{ij})^2=n\sum_{i\kappa} (h^0_{ik})^2,$$
	that is
	$$(n-2)\|\sff_{g_0}\|^2=0,$$
	and a contradiction arises.

\medskip

Finally, let $(e_1,\ldots,e_{n-1})$ be the canonical base in $\R^{n-1}$. We will show that for every $1\leq r\leq n-1$, there exists $k_r$ such that $D_kF(p_0,k_0)[k_r]=e_r$. In Fermi coordinates, this is equivalent to prove that, given $r$, there exists $k_r\in T_{k_0}N$ such that
\begin{equation*}
\begin{split}
D_k\de_r \norm{\sff_{g+k}}^2\vert_{(p_0,k_0)}[k_r]&=1, \text{and} \\
D_k\de_s \norm{\sff_{g+k}}^2\vert_{(p_0,k_0)}[k_r]&=0 \hsp\text{for every}\hsp s\neq r.
\end{split}
\end{equation*} 
Let $i^*$ and $j^*$ be the indexes given by claim \eqref{c1}, and choose $k_r$ of the form \eqref{k} with $$c_{ij}(x)=\frac{-\delta^{ii^*}\delta^{jj^*}x_r}{2\left(\sum_{k}h^0_{ik}h^0_{jk} -2(h^0_{ij})^2\right)}.$$
It is easy to check that $k_r$ satisfies \eqref{prop-k1} with
\begin{equation}\label{der-c}
	\de_sc_{ij}(0)=\frac{-\delta^{ii^*}\delta^{jj^*}\delta^{rs}}{2\left(\sum_{k}h^0_{ik}h^0_{jk} -2(h^0_{ij})^2\right)}.
\end{equation}
The result follows from inserting \eqref{der-c} into \eqref{simpl}.
\end{proof}
\begin{proposition}\label{sur-G} For every $(p_0,u_0)\in G^{-1}({0})$, the map 
	\begin{equation*}
		\begin{split}
			DG(p_0,u_0):\R^{n-1}\times T_{u_0}\Theta &\to \R^{n-1} \\
			(p,v)&\mapsto D_uG(p_0,u_0)[v]+D_xG(p_0,u_0)[x]  
		\end{split}
	\end{equation*}
is surjective. 
\end{proposition}
\begin{proof}
By direct application of \eqref{norm-sff-conf},
	\begin{equation}\label{mappa2}
		\begin{split}
			G(p,u)= \nabla x\left(e^{-2u}\norm{\sff_{\bar g}(p)}^2\right)=-2e^{-2u}\nabla_xu\norm{\sff(p)}^2+e^{-2u}\nabla_x\norm{\sff(p)}^2
		\end{split} 
	\end{equation}
Fix $0\leq s \leq n-1$. We have:
\begin{align*}
&D_u\de_s \left.\norm{\sff_{e^{2u}\bar g}(p)}^2\right\vert_{\substack{u=u_0 \\ p=p_0}}[v] = \left.\frac{\de }{\de t}\right\vert_{\substack{t=0 \\ p=p_0}}\bigg[-2e^{-2(u_0+tv)}\de_s(u_0+tv)\norm{\sff(p)}^2+e^{-2(u_0+tv)}\de_s\norm{\sff(p)}^2\bigg] \\ &=\bigg(4\de_su_0(p_0)e^{-2u_0(p_0)}\norm{\sff(p_0)}^2-2e^{2u_0(p_0)}\de_s\norm{\sff(p_0)}^2\bigg)v(p_0)-2e^{-2u_0(p_0)}\norm{\sff(p_0)}^2\de_sv(p_0) \\ &= -2\de_s\left(e^{-2u_0}\norm{\sff}^2\right)(p_0)-2e^{-2u_0(p_0)}\norm{\sff(p_0)}^2\de_sv(p_0) \\ &=-2e^{-2u_0(p_0)}\norm{\sff(p_0)}^2\de_sv(p_0)
\end{align*}
We will show that, for every $1\leq r\leq n-1$, there exists $v_r\in T_{u_0}\Theta$ such that
\begin{equation}\label{surj-v}
	\begin{split}
	D_u\de_r \left.\norm{\sff_{e^{2u}\bar g}(p)}^2\right\vert_{\substack{u=u_0 \\ p=p_0}}[v_r]&=1 \\
	D_u\de_s \left.\norm{\sff_{e^{2u}\bar g}(p)}^2\right\vert_{\substack{u=u_0 \\ p=p_0}}[v_r] &= 0\hsp \text{for every}\hsp s\neq r.
	\end{split}
\end{equation}
The first equation is equivalent to find $v$ such that
\begin{equation*}
-2e^{-2u_0(p_0)}\norm{\sff(p_0)}^2\de_sv(p_0) = 1.
\end{equation*}
A simple solution to this equation is given by
%\begin{equation}
%	v(x,t)=-\frac{1}{2\norm{\sff(p_0)}^2}\int_0^{x_r}e^{2u_0(x_1,\ldots,x_{r-1},\tau,x_{r+1},\ldots,x_{n-1},t)}d\tau.
%\end{equation}
\begin{equation}
	v(x,t)=\frac{-x_r}{2e^{eu_0(p_0)}\norm{\sff(p_0)}^2}
\end{equation}
Observe that $v\in T_{u_0}\Theta$ since
\begin{align*}
\frac{\de v}{\de \nu} = -\left.\frac{\de v(x,t)}{\de t}\right\vert_{t=0} = 0.
\end{align*}
Moreover,
\begin{align*}
\de_s v(x,t)\big\vert_{(x,t)=0} = 0,
\end{align*}
for every $s\neq r$, so it safisfies \eqref{surj-v}.
\end{proof}

\end{document}